\documentclass{article}%
\usepackage{amsfonts}
\usepackage{graphicx}
\usepackage{amsmath}%
\usepackage{amssymb}
\providecommand{\U}[1]{\protect\rule{.1in}{.1in}}
\newtheorem{theorem}{Theorem}

\newtheorem{definition}[theorem]{Definition}
\newtheorem{example}[theorem]{Example}

\newtheorem{notation}[theorem]{Notation}

\newtheorem{proposition}[theorem]{Proposition}

\newenvironment{proof}[1][Proof]{\textbf{#1.} }{\ \rule{0.5em}{0.5em}}
\begin{document}

\title{Axiomatic Differential Geometry III-1\\-Its Landscape-\\Chapter 1: Model Theory I}
\author{Hirokazu NISHIMURA\\Institute of Mathematics\\University of Tsukuba\\Tsukuba, Ibaraki, 305-8571, JAPAN}
\maketitle

\begin{abstract}
In this paper is proposed a kind of model theory for our axiomatic
differential geometry. It is claimed that smooth manifolds, which have
occupied the center stage in differential geometry, should be replaced by
functors on the category of Weil algebras. Our model theory is geometrically
natural and conceptually motivated, while the model theory of synthetic
differential geometry is highly artificial and exquisitely technical.

\end{abstract}

\section{\label{s1}Introduction}

Consider two smooth curves sharing a unique point, say the origin, in the
$2$-dimensional Euclidean space. This is a familiar situation in high school
mathematics. From the authentic viewpoint of contemporary mathematics, whether
the two curves are transversal at the common point or they are tangential
therein, their intersection is the shared point, the same figure consisting of
a unique point. Now the story is over to most of the schoolboys and
schoolgirls. However, a good mathematician endowed with a cornucopia of
geometric acumen and academically sophisticated instinct feels a flavor of
impropriety here. What is wrong ?

To resolve the above paradoxical situation, \textit{synthetic differential
geometers} resort to the resurrection of \textit{nilpotent infinitesimals},
which were abandoned as anathema and replaced by so-called $\varepsilon
-\delta$ arguments in the 19th century. They insist that, were one able to
recognize nilpotent infinitesimals by looking closer and closer and using
highly sensitive microscopes if necessay, he or she would find distinct
intersections at the above two contradistinctive cases. The intersection of
the transversal curves is really one point, while that of the tangential
curves would be that point accompanied by its microcosm of numerous lurking
points infinitesimally close to it. In any case, synthetic differential
geometers were forced to invent an artifact, called \textit{well-adapted
models}, in which they are generously able to indulge in their favorite
nilpotent infinitesimals. For synthetic differential geometry, the reader is
referred to \cite{kock} and \cite{lav}.

Our solution to the above paradox is more realistic and highly geometrical. We
plumb impropriety in the definition of a \textit{figure}. It is nothing but
platitudinous to say that \textit{geometry is the study of figures}, but the
notion of a figure in mathematics has changed dramatically several times since
the days of ancient Greeks such as Euclid and Pythagoras. We ask again what a
figure is, as Riemann did at his trial lecture for his habilitation (entitled
''The hypotheses on which geometry is based''). We insist that the definition
of a figure should be hierarchical. The figure is to be depicted not only at
its $0$-th order level corresponding to the Weil algebra $\mathbf{R}$ (the
degenerate Weil algebra of real numbers) but at various levels corresponding
to various Weil algebras. Consider the two curves tangential at the unique
common point as above, whose intersection at the $0$-th order level is surely
one point. The depiction of each curve at the $1$-st order level corresponding
to the Weil algebra $\mathbf{R}\left[  x\right]  /\left(  x^{2}\right)  $ is
its tangent bundle, and the intersection of the two curves at the $1$-st order
level is the same tangent space of both curves to the tangential point. We
note in passing that if the intersection of two figures should always be a
figure, this example forces us to admit a figure whose $0$-th order
description is one point but whose $1$-st order level description is a
one-dimensional linear space over $\mathbf{R}$. We note also that the
depiction of a figure at the $0$-th order level does not determine its
depiction at the $1$-st order level uniquely. Formally speaking from a coign
of vantage of category theory, we propose that a \textit{figure} is a functor
on the category of Weil algebras.

The principal objective in this paper is to give a model theory to the
axiomatics in \cite{nishi3} by exploiting the above notion of a figure. The
model theory is explained in \S \ref{s4}. We review our axiomatics of
\cite{nishi3}\ in \S \ref{s3}. We give some preliminaries on Weil algebras and
convenient categories in \S \ref{s2}.

\section{\label{s2}Preliminaries}

\subsection{Weil Algebras}

Let $k$\ be a commutative ring. The category of Weil algebras over $k$\ (also
called Weil $k$-algebras) is denoted by $\mathbf{Weil}_{k}$. It is well known
that the category $\mathbf{Weil}_{k}$\ is left exact. The initial and terminal
object in $\mathbf{Weil}_{k}$\ is $k$ itself. Given two objects $W_{1}$ and
$W_{2}$ in the category $\mathbf{Weil}_{k}$, we denote their tensor algebra by
$W_{1}\otimes_{k}W_{2}$. For a good treatise on Weil algebras, the reader is
referred to \S \ 1.16 of \cite{kock}. Given a left exact category
$\mathcal{K}$\ and a $k$-algebra object $\mathbb{R}$\ in $\mathcal{K}$, there
is a canonical functor $\mathbb{R}\underline{\otimes}_{k}\cdot$\ (denoted by
$\mathbb{R\otimes\cdot}$\ in \cite{kock}) from the category $\mathbf{Weil}%
_{k}$ to the category of $k$-algebra objects and their homomorphisms in
$\mathcal{K}$.

\subsection{Convenient Categories}

The category of topological spaces and continuous mappins is by no means
cartesian closed. In 1967 Steenrod \cite{st} popularized the idea of
\textit{convenient category} by announcing that the category of compactly
generated spaces and continuous mappings renders a good setting for algebraic
topology. The proposed category is cartesian closed, complete and cocomplete,
and contains all CW complexes.

About the same time, an attempt to give a convenient category to smooth spaces
began, and we have a few candidates. For a thorough study on the relationship
among these proposed candidates, the reader is referred to \cite{stacey}, in
which the reader finds by way of example that the category of Fr\"{o}licher
spaces is a full subcategory of that of Souriau spaces, and the category of
Souriau spaces is in turn a full subcategory of Chen spaces. We have no
intention to discuss which is the best convenient category of smooth spaces
here. We content ourselves with denoting some of such categories by
$\mathbf{Smooth}$, which is required to be complete and cartesian closed at
least containing the category $\mathbf{Mf}$ of smooth manifolds as a full subcategory.

\section{\label{s3}The Axiomatics}

We review the axiomatics in \cite{nishi3}.

\begin{definition}
A \underline{DG-category} (DG stands for Differential Geometry) is a quadruple
$\left(  \mathcal{K},\mathbb{R},\mathbf{T},\alpha\right)  $, where

\begin{enumerate}
\item $\mathcal{K}$ is a category which is left exact and cartesian closed.

\item $\mathbb{R}$ is a commutative $k$-algebra object in $\mathcal{K}$.

\item Given a Weil $k$-algebra $W$, $\mathbf{T}^{W}:\mathcal{K}\rightarrow
\mathcal{K}$ is a left exact functor for any Weil $k$-algebra $W$ subject to
the condition that $\mathbf{T}^{k}:\mathcal{K}\rightarrow\mathcal{K}$ is the
identity functor, while we have
\[
\mathbf{T}^{W_{2}}\circ\mathbf{T}^{W_{1}}=\mathbf{T}^{W_{1}\otimes_{k}W_{2}}%
\]
for any Weil $k$-algebras $W_{1}$ and $W_{2}$.

\item Given a Weil $k$-algebra $W$, we have
\[
\mathbf{T}^{W}\mathbb{R}=\mathbb{R}\underline{\otimes}_{k}W
\]

\item $\alpha_{\varphi}:\mathbf{T}^{W_{1}}\overset{\cdot}{\rightarrow
}\mathbf{T}^{W_{2}}$ is a natural transformation for any morphism
$\varphi:W_{1}\rightarrow W_{2}$ in the category $\mathbf{Weil}_{k}$ such that
we have
\[
\alpha_{\psi}\cdot\alpha_{\varphi}=\alpha_{\psi\circ\varphi}%
\]
for any morphisms $\varphi:W_{1}\rightarrow W_{2}$ and $\psi:W_{2}\rightarrow
W_{3}$ in the category $\mathbf{Weil}_{k}$, while we have
\[
\alpha_{\mathrm{id}_{W}}=\mathrm{id}_{\mathbf{T}^{W}}%
\]
for any identity morphism $\mathrm{id}_{W}:W\rightarrow W$ in the category
$\mathbf{Weil}_{k}$.

\item Given a morphism $\varphi:W_{1}\rightarrow W_{2}$ in the category
$\mathbf{Weil}_{k}$, we have
\[
\alpha_{\varphi}\left(  \mathbb{R}\right)  =\mathbb{R}\underline{\otimes}%
_{k}\varphi
\]

\end{enumerate}
\end{definition}

\section{\label{s4}Model Theory}

Let $\mathbf{U}$\ be a complete and cartesian closed category with
$\mathbb{R}$\ being a $k$-algebra object in $\mathbf{U}$.

\begin{notation}
We denote by $\mathcal{K}_{\mathbf{U}}$\ the category whose objects are
functors from the category $\mathbf{Weil}_{k}$ to the category $\mathbf{U}$
and whose morphisms are their natural transformations.
\end{notation}

It is easy to see that

\begin{proposition}
The category $\mathcal{K}_{\mathbf{U}}$\ is complete and cartesian closed.
\end{proposition}

\begin{proof}
The proof is tremendously similar to that of the familiar fact that the arrow
category of a complete and cartesian closed category is complete and cartesian
closed. That the category $\mathcal{K}_{\mathbf{U}}$\ is complete follows from
Theorem 7.5.2 in \cite{sch}. The cartesian closedness of $\mathcal{K}%
_{\mathbf{U}}$\ is discussed in a subsequent paper, but the reader is referred
to Exercise 1.3.7 in \cite{jac} for the cartesian closedness of the arrow
category of a complete and cartesian closed category.
\end{proof}

\begin{notation}
Given an object $W$ in the category $\mathbf{Weil}_{k}$, we denote by
\[
\mathbf{T}_{\mathbf{U}}^{W}:\mathcal{K}_{\mathbf{U}}\rightarrow\mathcal{K}%
_{\mathbf{U}}%
\]
the functor obtained as the composition with the functor
\[
\_\otimes_{k}W:\mathbf{Weil}_{k}\rightarrow\mathbf{Weil}_{k}%
\]
so that for any object $M$\ in the category $\mathcal{K}_{\mathbf{U}}$, we
have
\[
\mathbf{T}_{\mathbf{U}}^{W}\left(  M\right)  =M\left(  \_\otimes_{k}W\right)
\]

\end{notation}

It is easy to see that

\begin{proposition}
We have
\[
\mathbf{T}_{\mathbf{U}}^{W_{2}}\circ\mathbf{T}_{\mathbf{U}}^{W_{1}}%
=\mathbf{T}_{\mathbf{U}}^{W_{1}\otimes W_{2}}%
\]
for any objects $W_{1},W_{2}$ in the category $\mathbf{Weil}_{k}$.
\end{proposition}

\begin{proof}
We have
\begin{align*}
\mathbf{T}_{\mathbf{U}}^{W_{2}}\circ\mathbf{T}_{\mathbf{U}}^{W_{1}}  &
=\left(  \_\otimes_{k}W_{1}\right)  \otimes_{k}W_{2}\\
& =\_\otimes_{k}\left(  W_{1}\otimes W_{2}\right) \\
& =\mathbf{T}_{\mathbf{U}}^{W_{1}\otimes W_{2}}%
\end{align*}

\end{proof}

It is also easy to see that

\begin{proposition}
The functor
\[
\mathbf{T}_{\mathbf{U}}^{W}:\mathcal{K}_{\mathbf{U}}\rightarrow\mathcal{K}%
_{\mathbf{U}}%
\]
preserves limits for each object $W$ in the category $\mathbf{Weil}_{k}$.
\end{proposition}

\begin{proof}
This follows easily from 7.5.2 and 7.5.3 in \cite{sch}.
\end{proof}

\begin{proposition}
Given a morphism $\varphi:W_{1}\rightarrow W_{2}$ in the category
$\mathbf{Weil}_{k}$ and an object $M$\ in the category $\mathcal{K}%
_{\mathbf{U}}$, the assignment of the morphism
\[
M\left(  W\otimes_{k}\varphi\right)  :M\left(  W\otimes_{k}W_{1}\right)
\rightarrow M\left(  W\otimes_{k}W_{2}\right)
\]
in the category $\mathbf{U}$ to each object $W$\ in the category
$\mathbf{Weil}_{k}$ is a morphism
\[
\mathbf{T}_{\mathbf{U}}^{W_{1}}\left(  M\right)  \rightarrow\mathbf{T}%
_{\mathbf{U}}^{W_{2}}\left(  M\right)
\]
in the category $\mathcal{K}_{\mathbf{U}}$, which we denote by $\alpha
_{\varphi}^{\mathbf{U}}\left(  M\right)  $.
\end{proposition}

\begin{proof}
Given a morphism $\psi:W\rightarrow W^{\prime}$ in the category $\mathbf{Weil}%
_{k}$, the diagram
\[%
\begin{array}
[c]{ccc}%
M\left(  W\otimes_{k}W_{1}\right)  & \underrightarrow{M\left(  W\otimes
_{k}\varphi\right)  } & M\left(  W\otimes_{k}W_{2}\right) \\
M\left(  \psi\otimes_{k}W_{1}\right)  \downarrow &  & \downarrow M\left(
\psi\otimes_{k}W_{2}\right) \\
M\left(  W^{\prime}\otimes_{k}W_{1}\right)  & \overrightarrow{M\left(
W^{\prime}\otimes_{k}\varphi\right)  } & M\left(  W^{\prime}\otimes_{k}%
W_{2}\right)
\end{array}
\]
is commutative, so that the desired conclusion follows.
\end{proof}

\begin{proposition}
Given a morphism $\varphi:W_{1}\rightarrow W_{2}$ in the category
$\mathbf{Weil}_{k}$, the assignment of the morphism
\[
\alpha_{\varphi}^{\mathbf{U}}\left(  M\right)  :\mathbf{T}_{\mathbf{U}}%
^{W_{1}}\left(  M\right)  \rightarrow\mathbf{T}_{\mathbf{U}}^{W_{2}}\left(
M\right)
\]
in the category $\mathcal{K}_{\mathbf{U}}$\ to each object $W$\ in the
category $\mathbf{Weil}_{k}$ is a natural transformation
\[
\mathbf{T}_{\mathbf{U}}^{W_{1}}\Rightarrow\mathbf{T}_{\mathbf{U}}^{W_{2}}%
\]
which we denote by $\alpha_{\varphi}^{\mathbf{U}}$.
\end{proposition}

\begin{proof}
Given a morphism $f:M_{1}\rightarrow M_{2}$ in the category $\mathcal{K}%
_{\mathbf{U}}$, the diagram
\[%
\begin{array}
[c]{ccc}%
M_{1}\left(  W\otimes_{k}W_{1}\right)  & \underrightarrow{M_{1}\left(
W\otimes_{k}\varphi\right)  } & M_{1}\left(  W\otimes_{k}W_{2}\right) \\
f_{W\otimes_{k}W_{1}}\downarrow &  & \downarrow f_{W\otimes_{k}W_{2}}\\
M_{2}\left(  W\otimes_{k}W_{1}\right)  & \overrightarrow{M_{2}\left(
W\otimes_{k}\varphi\right)  } & M_{2}\left(  W\otimes_{k}W_{2}\right)
\end{array}
\]
is commutative, so that the desired conclusion follows.
\end{proof}

It is easy to see that

\begin{proposition}
We have
\[
\alpha_{\psi}^{\mathbf{U}}\circ\alpha_{\varphi}^{\mathbf{U}}=\alpha_{\psi
\circ\varphi}^{\mathbf{U}}%
\]
for any morhisms $\varphi:W_{1}\rightarrow W_{2}$ and $\psi:W_{2}\rightarrow
W_{3}$ in the category $\mathbf{Weil}_{k}$.
\end{proposition}

\begin{proof}
Given an object $M$\ in the category $\mathcal{K}_{\mathbf{U}}$, we have
\[
M\left(  \_\otimes_{k}\psi\right)  \circ M\left(  \_\otimes_{k}\varphi\right)
=M\left(  \_\otimes_{k}\left(  \psi\circ\varphi\right)  \right)
\]
so that the desired conclusion follows.
\end{proof}

\begin{notation}
We denote by $\mathbb{R}_{\mathbf{U}}$ the functor
\[
\mathbb{R}\underline{\mathbb{\otimes}}_{k}\_:\mathbf{Weil}_{k}\rightarrow
\mathbf{U}%
\]

\end{notation}

It is easy to see that

\begin{proposition}
We have
\[
\mathbf{T}_{\mathbf{U}}^{W}\left(  \mathbb{R}_{\mathbf{U}}\right)
=\mathbb{R}_{\mathbf{U}}\underline{\otimes}_{k}W
\]
for any object $W$\ in the category $\mathbf{Weil}_{k}$.
\end{proposition}

It is also easy to see that

\begin{proposition}
We have
\[
\alpha_{\varphi}^{\mathbf{U}}\left(  \mathbb{R}_{\mathbf{U}}\right)
=\mathbb{R}_{\mathbf{U}}\underline{\otimes}_{k}\varphi
\]
for any morphism $\varphi:W_{1}\rightarrow W_{2}$ in the category
$\mathbf{Weil}_{k}$.
\end{proposition}

Now we recapitulate as follows.

\begin{theorem}
The quadruple
\[
\left(  \mathcal{K}_{\mathbf{U}},\mathbb{R}_{\mathbf{U}},\mathbf{T}%
_{\mathbf{U}},\alpha^{\mathbf{U}}\right)
\]
is a DG-category.
\end{theorem}

\begin{example}
Let $\mathbf{U}=\mathbf{Smooth}$ with $k=\mathbf{R}$ and $\mathbb{R}%
=\mathbf{R}$.We denote by $\mathbf{Mf}$\ the category of smooth manifolds,
which can be regarded as a subcategory of $\mathbf{Smooth}$. It is well known
(cf. Theorem 31.7 in \cite{kri}) that there is a bifunctor
\[
T:\mathbf{Weil}_{\mathbf{R}}\times\mathbf{Mf}\rightarrow\mathbf{Mf}%
\]
Therefore each smooth manifold $M$\ can be regarded as the functor
\[
T(\_,M):\mathbf{Weil}_{\mathbf{R}}\rightarrow\mathbf{Smooth}%
\]
which is an object in $\mathcal{K}_{\mathbf{Smooth}}$. This gives rise to a
functor from the category $\mathbf{Mf}$\ to the category $\mathcal{K}%
_{\mathbf{Smooth}}$.
\end{example}

\end{document}